\documentclass[a4paper, tikz]{article}

\usepackage[english]{babel}
\usepackage[utf8]{inputenc}
\usepackage{amsmath, amsthm, amsfonts, amssymb, comment}
\usepackage{graphicx}
\usepackage[colorinlistoftodos]{todonotes}

\usepackage{forest}

\usepackage{hyperref}

\usepackage{csquotes}

\usetikzlibrary{matrix}

\newtheorem*{thm}{Theorem}
\newtheorem{prop}{Proposition}
\newtheorem*{cor}{Corollary}
\newtheorem*{rem}{Remark}

\usepackage{graphicx}
\usepackage{caption}
\usepackage{subcaption}

\title{On the monodromy of almost toric fibrations on the complex projective plane}

\date{}

\begin{document}
\begin{center}

\Large{{\bf On the monodromy of almost toric fibrations on the complex projective plane}}
\vspace{5mm}

\normalsize

{\bf Gleb Smirnov}\footnote{SISSA and Moscow State University. E-mail: gsmirnov@sissa.it
}
\end{center}

\begin{abstract}
We describe the monodromy for almost toric Lagrangian fibrations on the complex projective plane $\mathbb{CP}^2$.
\end{abstract}

\section{Introduction}

Let $(M, \omega)$ be an almost toric 4-manifold in the sense of Symington and Leung (see \cite{LS}), i.e. a closed symplectic manifold admitting an almost toric fibration. Recall that an almost toric fibration is a Lagrangian torus fibration
$$
\displaystyle{\pi\colon M \mapsto B}
$$
such that the fibers have only focus-focus or elliptic singularities.

In \cite{LS}, Leung and Symington gave a complete classification of the total spaces and bases of almost toric fibrations. They also formulated the classification problem, up to a fiber preserving symplectomorphism, for those fibrations.

Let us denote by $B_0$ the set of regular values of the map $\pi$, i.e. the part of $B$ on which $\pi$ defines a torus bundle.

We shall say that the preimage $M_0 = \pi^{-1}(B_0)$ is the \emph{regular part} of $\pi$.
The set of all critical values of $\pi$ is called the \emph{bifurcation diagram} of $\pi$.

Keeping in mind the classification problem for almost toric fibrations, we are studying the topology of the bundle $\pi \colon M_0 \mapsto B_0$. The structure of this bundle is completely determined by its monodromy.

Perhaps, the simplest example of an almost toric manifold is the complex projective plane $\mathbb{CP}^2$. In this case, possible bases of an almost toric fibration are described by the following statement. 

\begin{thm}[\cite{LS}]
There exist precisely four distinct bases for an almost toric fibration on the complex projective plane. They are shown in Figure \ref{fig:four_bases}. 
\begin{figure}[h]
        \centering
        \begin{subfigure}[b]{0.1\textwidth}
                \includegraphics[width=\textwidth]{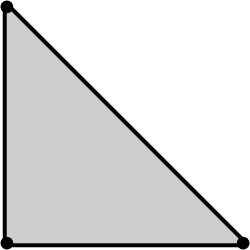}
                \caption{}
                \label{fig:four_bases_0}
        \end{subfigure}\qquad %
        ~ %add desired spacing between images, e. g. ~, \quad, \qquad, \hfill etc.
        \begin{subfigure}[b]{0.1\textwidth}
                \includegraphics[width=\textwidth]{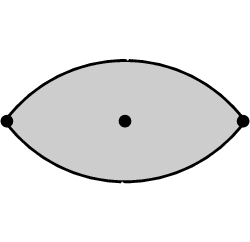}
                \caption{}
                \label{fig:four_bases_1}
        \end{subfigure}\qquad 
        ~ %add desired spacing between images, e. g. ~, \quad, \qquad, \hfill etc.
        \begin{subfigure}[b]{0.1\textwidth}
                \includegraphics[width=\textwidth]{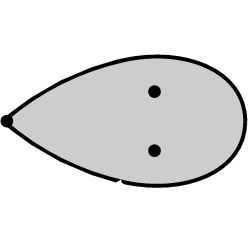}
                \caption{}
                \label{fig:four_bases_2}
        \end{subfigure}\qquad 
        \begin{subfigure}[b]{0.1\textwidth}
                \includegraphics[width=\textwidth]{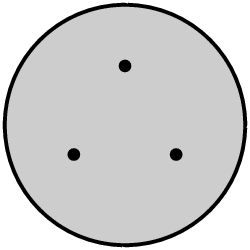}
                \caption{}
                \label{fig:four_bases_3}
        \end{subfigure}
        \caption{Four bases of almost toric fibrations on $\mathbb{CP}^2$.}\label{fig:four_bases}
\end{figure}
\end{thm}
In fact, as it is shown by Oshemkov \cite{O11}, this theorem also holds for Lagrangian fibrations with non-degenerate singularities.

Note that the first two diagrams also arise in mechanics as iso-energy
diagrams in the irreducible integrable problem of motion of the Kowalevski gyrostat in a double force field \cite{Kh2008, KhHMJ}.

Recall that isolated critical values are the images of focus-focus singular points. One can observe that a base is uniquely determined by the number of focus-focus points.

For the first base (case \ref{fig:four_bases_0}) containing no focus-focus points, the fibration $\pi$ is just a trivial torus bundle over a disk.

For only one focus-focus point (case \ref{fig:four_bases_1}), the topology of $\pi$ is also known: the base $B_0$ is a punctured disk and the structure of the corresponding torus bundle is determined by a single matrix. It is well known (see, for example, \cite{LU1, M, Z1997, BO}) that this matrix is conjugate 
to the following one 
\begin{equation}\label{eq1:focus_monodromy}
\displaystyle{
\begin{pmatrix}
1 & 1\\
0 & 1
\end{pmatrix}.
}
\end{equation}
In this paper, we consider the remaining cases of two and three focus-focus points.

\section{Background:\,toric and almost toric fibration}

A \emph{Lagrangian torus fibration} on a symplectic 4-manifold $(M,\omega)$ is a smooth map $\pi\colon M \mapsto B$ to a space of dimension 2 such that the preimages of regular values of $\pi$ are Lagrangian tori (half dimensional tori on which the symplectic structure vanishes). The projection map $\pi$ also is called the \emph{moment map}.

An \emph{almost toric fibration} on $M$ is a Lagrangian torus fibration having only elliptic or focus-focus singularities, i.e. any point of $M$ has a symplectic neighbourhood ($\omega = \sum_i \text{d}\,p_i \wedge \text{d}\,q_i$) in which the projection map $\pi$ has one of the following forms:
$$
\left.
\begin{aligned}
\pi(p,q)
&= (p_1,p_2),\quad \text{\emph{regular point,}}\\
\pi(p,q)
&= (p_1,p_2^2 + q_2^2),\quad \text{\emph{elliptic rank 1,}}\\
\pi(p,q)
&= (p_1^2 + q_1^2, p_2^2 + q_2^2),\quad \text{\emph{elliptic rank 0,}}\\
\pi(p,q)
&= (p_1 q_1 + p_2 q_2, p_1 q_2 - p_2 q_1),\quad \text{\emph{focus-focus.}}
\end{aligned}
\right.
$$
An \emph{almost toric manifold} is a closed symplectic 4-manifold equipped with an almost toric fibration. The list of all almost toric 4-\emph{manifolds} can be found in \cite{LS}. However the classification problem for almost toric \emph{fibrations} up to a fiberwise symplectomorphism is still unsolved.

In symplectic geometry, a \emph{toric 4-manifold} is a closed symplectic 4-manifold with an effective Hamiltonian torus action; the torus action generates a Lagrangian torus fibration that is called a \emph{toric fibration}. Any singular point of a torus fibration is an elliptic type point. Toric fibrations are completely classified by Delzant's theorem (see, for example, \cite{Delzant1988, cdS}): the total space, symplectic structure and torus action are completely determined by the polytope in $\mathbb{R}^2$ that is the image of the moment map.

Let $\mathbb{T}^2 = \left\{(\varphi_1, \varphi_2):\varphi_i\,\text{mod}\,2\pi\right\}$ be a torus acting on $\mathbb{CP}^2$ by
$$
(\varphi_1,\varphi_2)[z_1:z_2:z_3] \mapsto [z_1 e^{i \varphi_1}:z_2 e^{i \varphi_2}:z_3].
$$
It can be easily checked that this action is Hamiltonian; the image of the moment map is shown in Figure \ref{fig:four_bases_0}.

Now we shall give an example of an almost toric fibration on $\mathbb{CP}^2$ that is not toric (even locally). Let $M = \mathbb{C}^3$, $\omega = \frac{i}{2} \sum_k \text{d}\,z_k \wedge \text{d}\,\bar{z}_k$ be a symplectic vector space. Consider the function $F = \sum_k |z_k|^2$ that is invariant under the following Hamiltonian $S^1$-action
$$
\displaystyle{ [z_1:z_2:z_3] \mapsto e^{it}[z_1:z_2:z_3]. }
$$
The submanifold $F^{-1}(1)$ is diffeomorphic to $S^5$;
the quotient space $S^5/S^1 = \mathbb{CP}^2$ has the induced symplectic structure.
Let us consider all quadratic functions on $\mathbb{C}^3$ of the form $\sum a_{ij}z_i \bar{z}_j$ which are invariant under this $S^1$-action. These functions generate the algebra $u(3)$ and generators have the form
$$
\displaystyle{h_1 = z_1 \bar{z}_1,\quad h_2 = z_2 \bar{z}_2,\quad h_3 = z_3 \bar{z}_3,}
$$
$$
\displaystyle{f_1 = z_2 \bar{z}_3 + z_3 \bar{z}_2,\quad f_2 = z_1 \bar{z}_3 + z_3 \bar{z}_1,\quad f_3 = z_1 \bar{z}_2 + z_2 \bar{z}_1,}
$$
$$
\displaystyle{
m_1 = z_2 \bar{z}_3 - z_3 \bar{z}_2,\quad 
m_2 = z_3 \bar{z}_1 - z_1 \bar{z}_3,\quad
m_3 = z_1 \bar{z}_2 - z_2 \bar{z}_1.
}
$$
It can be shown by direct calculation that the pair of functions
$$
\displaystyle{ H = h_2 - h_3,\quad G = f_3 f_2 - m_3 m_2 }
$$
defines an almost toric fibration; the image of the moment map $(H,G)\colon \mathbb{CP}^2 \mapsto \mathbb{R}^2$ shown in Figure \ref{fig:four_bases_1}.

This fibration has a point of focus-focus type. This implies that it cannot be toric. Nevertheless it admits the following Hamiltonian fiberwise $S^1$-action
$$
\displaystyle{[z_1:z_2:z_3] \mapsto [z_1:z_2 e^{i\psi}:z_3 e^{-i\psi}].}
$$
If there exists a fiberwise Hamiltonian $S^1$-action on a Lagrangian fibration, then this fibration is called \emph{semi-toric}. Semi-toric fibrations are classified by Pelayo and V\~u Ng\d{o}c \cite{P}. Note that in both examples considered above, the fibrations are semi-toric. However, we shall see that Lagrangian fibrations having two or three focus-focus points (cases \ref{fig:four_bases_2} and \ref{fig:four_bases_3}) cannot be semi-toric.

If we remove all singular fibers from a Lagrangian fibration, then what remains is a Lagrangian torus bundle. Following Duistermaat \cite{D}, we consider the \emph{monodromy} of a Lagrangian bundle. Let $\pi\colon M\mapsto B$ be a Lagrangian torus bundle and let $\gamma$ be a simple closed path in $B$. The preimage $Q = \pi^{-1}(\gamma)$ is a torus bundle over the circle. Clearly, $Q$ can be represented as the result of identification of the boundary tori of the 3-cylinder $\mathbb{T}^2 \times [0,1]$ by some homeomorphism.
This gluing homeomorphism is uniquely defined (up to an isotopy) by an integer unimodular matrix that is called the \emph{monodromy matrix}. This matrix depends of course on the choice of a basis on the tori, but its conjugacy class is well defined and is a complete invariant of the bundle $\pi\colon Q \mapsto \gamma$. In this paper we are mainly interested in the topology of Lagrangian bundles obtained from Lagrangian fibrations on $\mathbb{CP}^2$. 
Since their topology is completely determined by the monodromy, the problem of description of such bundles boils down to the description of possible monodromies.

\section{Description of the monodromy}

The aim of this paper is to prove the following statement.
\begin{thm}
The regular part of an almost toric fibration 
on the complex projective plane is determined by its base up to a fiber preserving homeomorphism; the monodromy is shown in Figure \ref{fig:mtr}.
\begin{figure}[h]
        \centering
        \begin{subfigure}[b]{0.4\textwidth}
                \includegraphics[width=\textwidth]{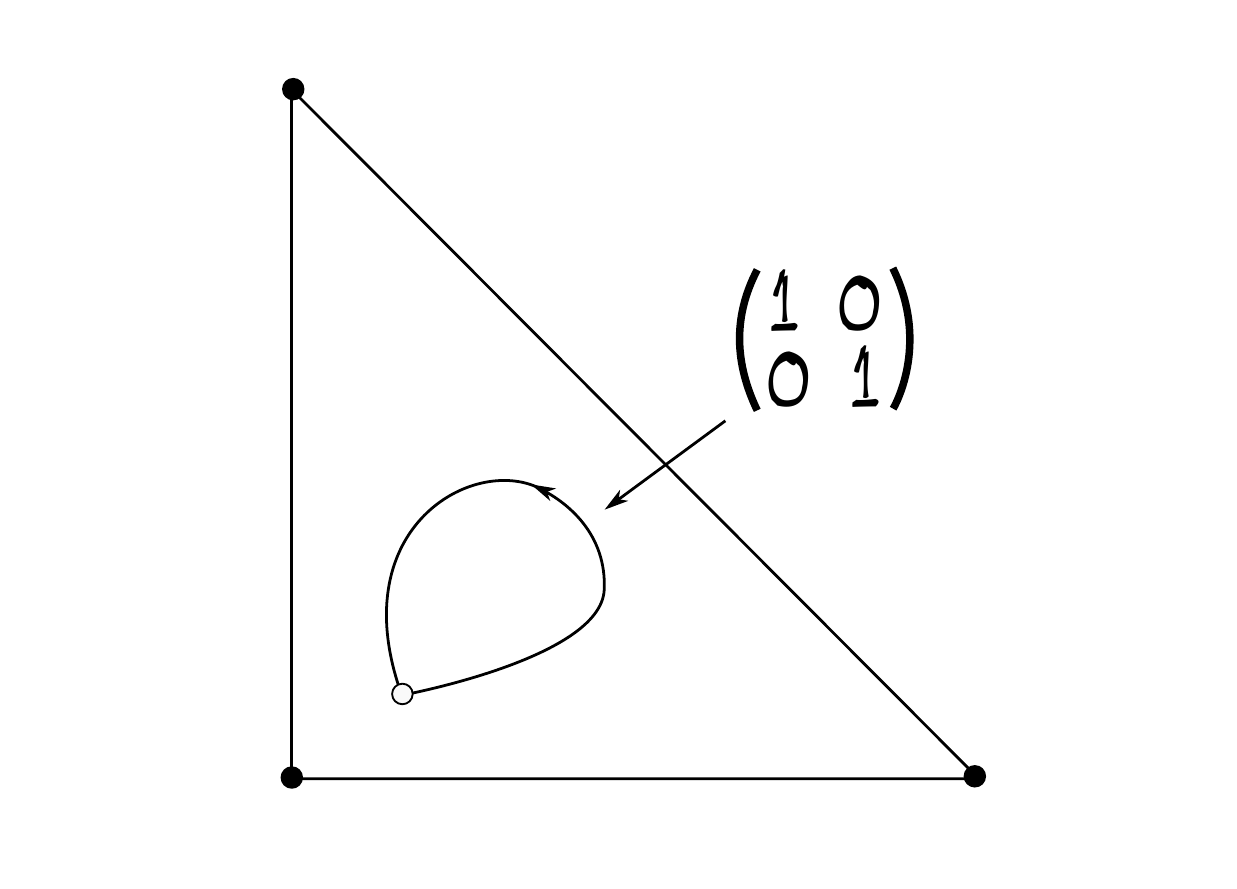}
                \caption{}

        \end{subfigure}\qquad
        \begin{subfigure}[b]{0.4\textwidth}
                \includegraphics[width=\textwidth]{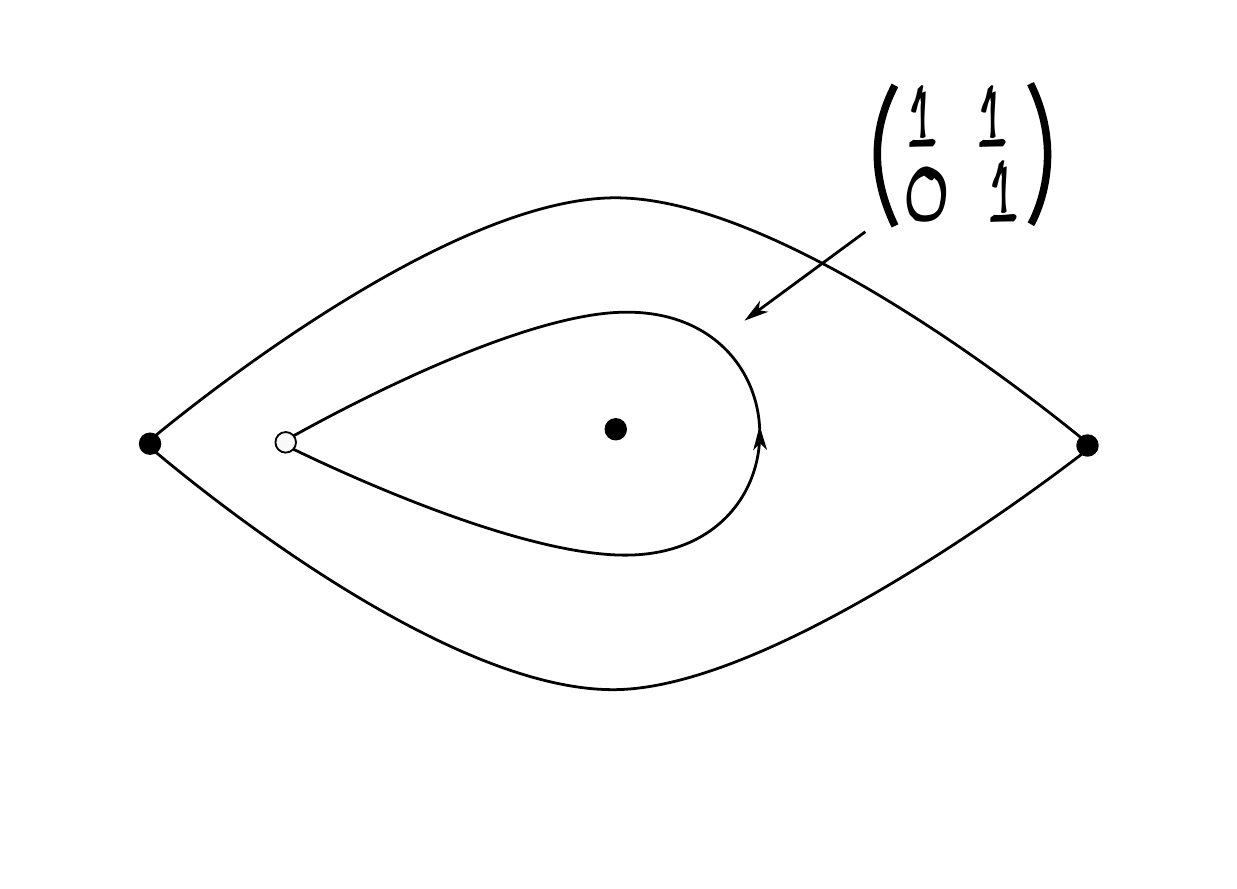}
                \caption{}

        \end{subfigure}\\

        \begin{subfigure}[b]{0.4\textwidth}
                \includegraphics[width=\textwidth]{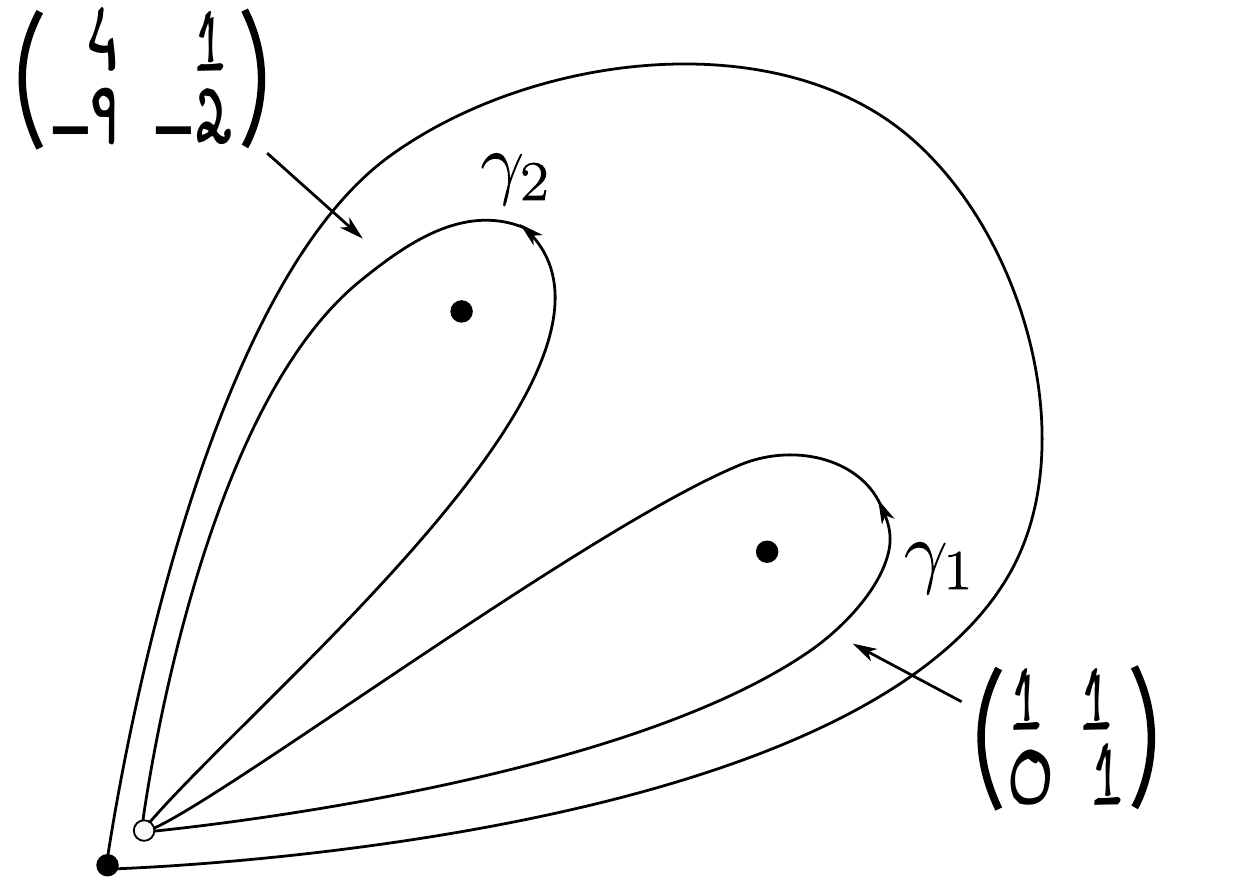}
                \caption{}

        \end{subfigure}\qquad %
        ~ %add desired spacing between images, e. g. ~, \quad, \qquad, \hfill etc.
        \begin{subfigure}[b]{0.5\textwidth}
                \includegraphics[width=\textwidth]{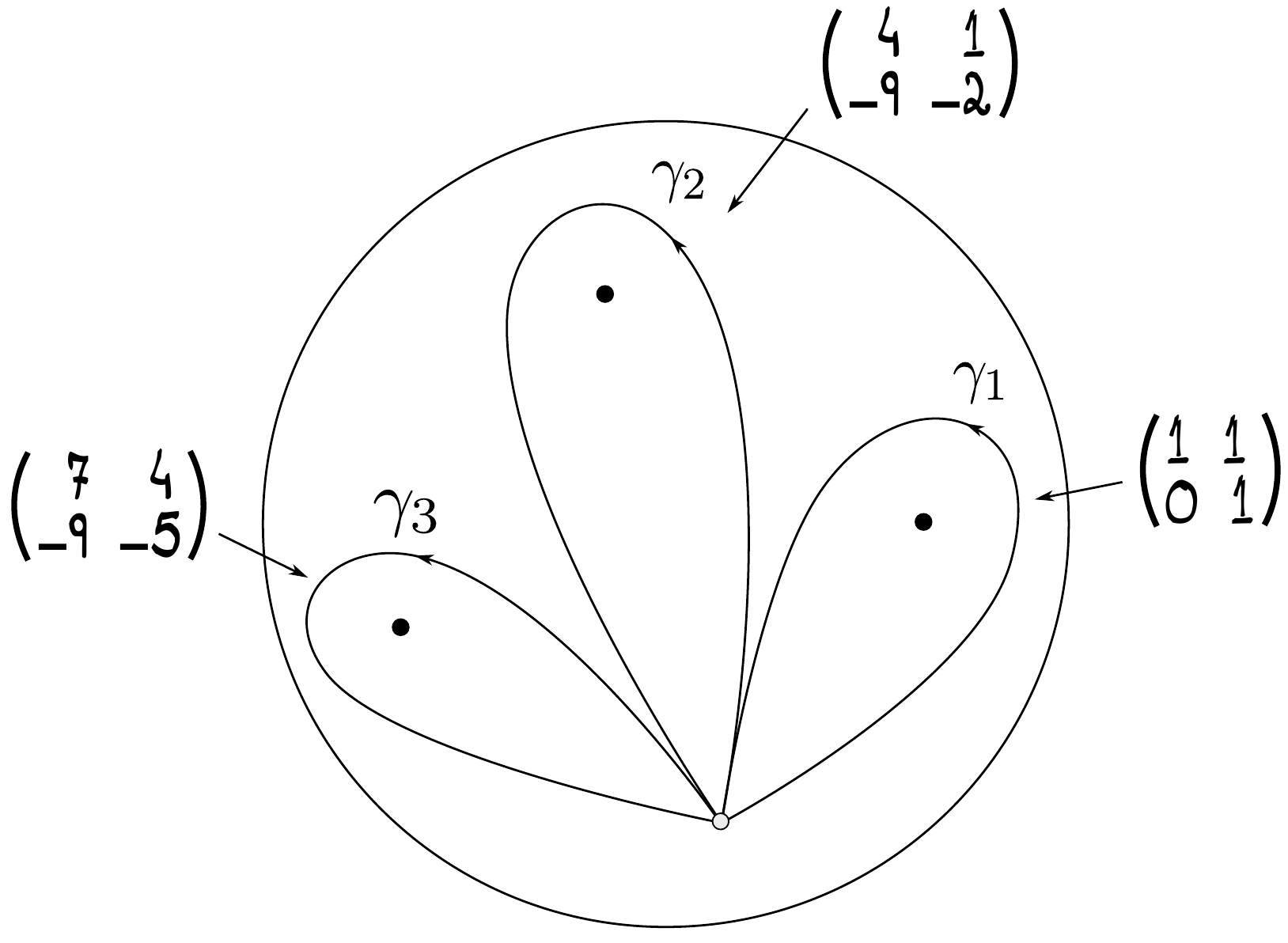}
                \caption{}
        \end{subfigure}
        \caption{Monodromy of fibrations on $\mathbb{CP}^2$.}\label{fig:mtr}
\end{figure}
\end{thm}
\begin{rem}
One may use Zung's approach \cite{Z2003} to prove that not only the regular part, but also the whole fibration is uniquely determined by its base.
\end{rem}

Let us divide the proof into several propositions.

Assume that $\pi \colon M \mapsto \mathbb{CP}^2$ is an almost toric fibration that contains two points of focus-focus type. Let $P \in B_0$ be a regular value of $\pi$, and let $\gamma$ be a closed simple path starting and ending at $P$. Suppose the path $\gamma$ goes around both focus-focus points\footnote{Such path $\gamma$ is unique up to orientation.}. For each focus-focus point consider a path $\gamma_i,\,i = 1,2$ going around this point, but not both points together. These paths may be chosen in such a way that $\gamma_1 \gamma_2 = \gamma$, see Figure \ref{fig:mtr}. 
By $M_1$ and $M_2$ denote the monodromy matrices corresponding to the paths $\gamma_1$ and $\gamma_2$, respectively.
\begin{prop}\label{thm:prop_2pts}
There exists a basis of cycles on the torus $\pi^{-1}(P)$ such that the monodromy matrices take one of the two following forms: 

\begin{enumerate}
\item $\displaystyle{ 
M_{1}^{(1)} = \begin{pmatrix} 3 & 1 \\ -4 & -1 \end{pmatrix}, 
\quad 
M_{2}^{(1)} = \begin{pmatrix} 3 & 4 \\ -1 & -1 \end{pmatrix}
},
$

\item $\displaystyle{ 
M_{1}^{(2)} = \begin{pmatrix} 6 & 1 \\ -25 & -4 \end{pmatrix},\quad 
M_{2}^{(2)} = \begin{pmatrix} 3 & 1 \\ -4 & -1 \end{pmatrix} 
}.$\\
Furthermore, the transformation of paths $(\gamma_1, \gamma_2) \mapsto~(\gamma_2, \gamma_2 \gamma_1 \gamma_{2}^{-1})$ takes the latter pair to the former one.
\end{enumerate}
Moreover, such a basis is unique up to a change of orientation of both cycles. 
\end{prop}
\begin{proof}
First we prove that there exists a basis such that the monodromy matrix corresponding to the path $\gamma$ has the form 
$$
\displaystyle{
M = \begin{pmatrix}
-7 & -1\\
1 & 0
\end{pmatrix}.
}
$$
Next we describe the set of all pairs $(M_1, M_2)$ satisfying
$$
\displaystyle{
M_2\,M_1 = M,
}
$$
where the matrices $M_i$ are conjugate to matrix \eqref{eq1:focus_monodromy}.

Finally we show that for any pair $(M_1, M_2)$ there exists a basis such that the pair takes one of the forms listed in the proposition.
\vspace{1cm}\\
\emph{1. Conjugacy class of $M$.}\\
Let us choose a convenient basis for each torus in a neighbourhood of an elliptic rank 0 point (see \cite{IGS, BO} for details on the topology of such singular points).

\begin{figure}
\centering
\includegraphics[width=0.3\textwidth]{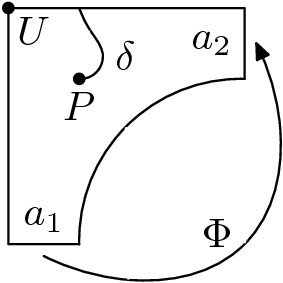}
\caption{\label{fig:elliptic} Neighbourhood of a rank 0 elliptic point.}
\end{figure}

Let $U$ be the image of a neighbourhood of an elliptic point under the map $\pi$, see Figure \ref{fig:elliptic}. Fix any point $P \in U$ and let $\delta$ be a path joining $P$ with the bifurcation diagram.
Denote by $\mathcal{A}$ the preimage $\pi^{-1}(\delta)$; $\mathcal{A}$ is a solid torus foliated into Lagrangian tori, i.e. it is a fibration generated by the function $f\colon D \times S^1 \mapsto \mathbb{C},\, f(z) = z \bar{z}$, where $z$ is a complex coordinate in the disk $D(z) = \{z|\,|z| \leq 1\}$; the fibration $\mathcal{A}$ is said to be the \emph{fibered solid torus}.
Since $\pi^{-1}(P)$ bounds the solid torus $\mathcal{A}$, its fundamental group has a \emph{marked cycle} $\theta$ that is contractible in $\mathcal{A}$.

Further, let $\delta_1$ and $\delta_2$ be paths joining the point $P$ to different edges of the bifurcation diagram; then it is possible to choose two cycles $(\theta_1, \theta_2)$ on the torus $\pi^{-1}(P)$ that are contractible in the corresponding solid torus $\mathcal{A}_1 = \pi^{-1}(\delta_1)$ and $\mathcal{A}_2 = \pi^{-1}(\delta_2)$, respectively.

Using the Eliasson theorem \cite{E, Mir}, it is easy to show that the intersection of these two cycles is a single point. Therefore the pair $(\theta_1, \theta_2)$ is a basis for $\pi^{-1}(P)$.
More precisely, it becomes a basis after we orient the cycles $\theta_1$, $\theta_2$.
The orientation is given in the following way. Consider the solid torus $\mathcal{A}_1$. The cycle $\theta_1$ is contractible in $\mathcal{A}_1$ and the cycle $\theta_2$ is homotopy equivalent to the boundary of the symplectic disk that is the preimage of the corresponding edge of the diagram. If we orient the disk by the symplectic form, we get the induced orientation on the boundary. Similarly, we can orient cycle $\theta_1$.
The basis $(\theta_1, \theta_2)$ is said to be the \emph{canonical basis}; it is well defined up to permutation of the cycles.

Now let us consider the boundary segments $a_1$ and $a_2$ shown in Figure \ref{fig:elliptic} and the corresponding fibrations $\pi^{-1}(a_1)$ and $\pi^{-1}(a_2)$; the spaces $\pi^{-1}(a_i)$ are fibered solid tori. Let us glue these fibrations by a homeomorphism. In order to describe all possible gluings up to a fiberwise isotopy, we choose a basis on the boundary torus of each fibration $\pi^{-1}(a_i)$ in the following way. The first cycle $\lambda_i$ is the meridian of the solid torus, i.e. this cycle is contractible in this solid torus. The second cycle $\mu_i$ is chosen in such a way that the pair $(\lambda_i, \mu_i)$ is a basis for the boundary torus.

Such a basis is not unique. However the canonical basis for each torus in $\pi^{-1}(U)$ has already been constructed. Set
$$
\displaystyle{ \lambda_1 = \theta_{1}, \mu_1 = \theta_2\  \text{\normalfont{and}}\ \lambda_2 = \theta_{2}, \mu_2 = \theta_1.}
$$

It is not hard to prove that the gluing map is completely defined up to a fiberwise isotopy by the mapping of a single torus, i.e. it is defined by the transition matrix from $(\lambda_1, \mu_1)$ to $(\lambda_2, \mu_2)$. Notice that we have the following additional restrictions on this matrix: 
\begin{itemize}
\item We should be careful about the gluing of 
the critical set of our Lagrangian fibration.
The preimage of the edges of the diagram are symplectic disks. We need to glue these disks along the boundary circles; the result of the gluing is the \emph{symplectic} sphere with a double point. Therefore the gluing map must reverse the orientations of the boundary circles. It follows that this map reverses the orientations of the singular fibers of $\pi^{-1}(a_1)$ and $\pi^{-1}(a_2)$;

\item The gluing map takes the cycle that is contractible in $\pi^{-1}(a_1)$ to the contractible cycle in $\pi^{-1}(a_2)$.
In fact, this map preserves the orientation of the contractible cycle.
\end{itemize}
The result is that the transition matrix has the form
$$
\displaystyle{
\begin{pmatrix}
\lambda_2\\
\mu_2
\end{pmatrix} = 
\begin{pmatrix}
1 & 0\\
-k & -1
\end{pmatrix}
\begin{pmatrix}
\lambda_1\\
\mu_1
\end{pmatrix}
}.
$$

The number $k$ has a clear geometric meaning. Let $S$ be the symplectic sphere that is the set of singular points of our Lagrangian fibration, then 
$$
\displaystyle{ [S]^2 = 2 + k,}
$$
where $[S]^2$ is the self-intersection index of $S$. If we orient the space by $\omega^2$, where $\omega$ is the symplectic form, then we may assume that $[S]^2$ is positive.

Oshemkov \cite[p.~63, Th.~30]{BO} proved that the homology class of the singular set of a Lagrangian fibration does not depend on the particular choice of a fibration, but only on the topology and symplectic structure of the ambient space. Considering the example of a toric fibration on $\mathbb{CP}^2$ above, we get $k = 7$.

So, we conclude that there exists a basis on the torus $\pi^{-1}(P)$ such that the monodromy matrix $M$ takes the form
$$
\displaystyle{
M = 
\begin{pmatrix}
0 & 1\\
1 & 0
\end{pmatrix}
\begin{pmatrix}
1 & 0\\
-7 & -1
\end{pmatrix} = 
\begin{pmatrix}
-7 & -1\\
1 & 0
\end{pmatrix}.
}
$$
\vspace{1cm}\\
\emph{2. Solutions of the equation $M_2 M_1 = M$.}\\
Since the matrices $M_i$ are conjugate to \eqref{eq1:focus_monodromy}, it follows that
$$
\displaystyle{
M_i = 
\begin{pmatrix}
1 & 0\\
0 & 1
\end{pmatrix} + \varepsilon_i \begin{pmatrix}
c_i d_i & d_i^2\\
-c_i^2 & -c_i d_i
\end{pmatrix},\quad \varepsilon_i = \pm 1,
}
$$
where the integers $c_i$ and $d_i$ are coprime; here $\varepsilon_i$ is the value of the determinant of a transition matrix to a basis in which the monodromy matrix $M_i$ takes form \eqref{eq1:focus_monodromy}.

Now, note that any focus-focus singularity gives rise to an orientation of the total space. Namely, let $\gamma$ be a small circle around the focus-focus point. Let us fix an orientation of the base, i.e. an orientation of our circle $\gamma$. Then one can orient the fibers by a basis of cycles such that the monodromy matrix corresponding to $\gamma$ equals \eqref{eq1:focus_monodromy}.
It is easy to prove that all bases such that the matrix is equal to \eqref{eq1:focus_monodromy}
have the same orientation.
Furthermore, if we change the orientation of $\gamma$, then we get the reversed orientation for the fibers.
The total space can be oriented by the sum of the orientations of the base and the fibers.
It was shown in \cite{CS} that all focus-focus points give rise to the same orientation. 
Consequently,
$$
\displaystyle{ \varepsilon_1 = \varepsilon_2 = \varepsilon = \pm 1.}
$$
This consideration implies that any solution of the system of equations 
$$
\displaystyle{
M_2 M_1 = \begin{pmatrix}
-7 & -1\\
1 & 0
\end{pmatrix}
}
$$
is either a solution to the following system
\begin{equation}\label{eq2:integer_system}
\displaystyle{
\left\{
\begin{aligned}
& -\varepsilon (c_1^2 + c_2^2) + 3\,c_1 c_2 = 1 \\
& -\varepsilon (d_1^2 + d_2^2) + 3\,d_1 d_2 = 1 \\
& -\varepsilon (c_1 d_1 + c_2 d_2) + \frac{3}{2} (c_1 d_2 + c_2 d_1) = \frac{7}{2} \\
& c_1 d_2 - c_2 d_1 = 3,
\end{aligned}
\right.
}
\end{equation}
or it can be transformed to be it by means of the transformation $(c_2,d_2) \mapsto (-c_2, -d_2)$; note that this transformation does not change $M_2$.

To solve system \eqref{eq2:integer_system}, let us consider the plane $\mathbb{R}^2$ with the inner product defined by the matrix
$$
\displaystyle{
\begin{pmatrix}
-\varepsilon & \dfrac{3}{2}\\
\dfrac{3}{2} & -\varepsilon
\end{pmatrix}
}.
$$
In terms of this inner product, we need to describe all unit integer vectors $(c_1, c_2)$ and $(d_1, d_2)$ with given area of the parallelogram spanned by these vectors and given angle between them.

The reader will have no difficulty in showing that for any pair $(d_1, d_2)$ satisfying \eqref{eq2:integer_system} there exists a unique pair $(c_1, c_2)$ that satisfies system \eqref{eq2:integer_system}: the solution is 
$$
\displaystyle{
c_1 = 8\,d_1 - 3\,\varepsilon \, d_2,\quad c_2 = -d_2 + 3\,\varepsilon\,d_1.
}
$$
Thus it remains to solve the equation
\begin{equation}\label{eq:integer_eq}
\displaystyle{ -\varepsilon\, (d_1^2 + d_2^2) + 3\, d_1 d_2 = 1.}
\end{equation}
In other words, we need to describe all integer points on the hyperbola.

Note that the set of the solutions for the case $\varepsilon = 1$ is mapped to the set of the solutions for $\varepsilon = -1$ by the following transformation 
\begin{equation}\label{eq:S}
\displaystyle{
\begin{pmatrix}
d_1\\
d_2
\end{pmatrix} \longmapsto
\begin{pmatrix}
2 & -1\\
-1 & 1
\end{pmatrix} \begin{pmatrix}
d_1\\
d_2
\end{pmatrix}
.}
\end{equation}
For this reason, it suffices to consider the case $\varepsilon = 1$.

The solutions of equation \eqref{eq:integer_eq} for $\varepsilon = 1$ have the form
$$
\displaystyle{\pm Z^n \boldsymbol{d}^{(i)},\quad Z = \begin{pmatrix}
21 & -8\\
8 & -3
\end{pmatrix},
}
$$
where
$$
\displaystyle{
\boldsymbol{d}^{(1)} = 
\begin{pmatrix}
1\\ 1
\end{pmatrix},\quad
\boldsymbol{d}^{(2)} = 
\begin{pmatrix}
1\\ 2
\end{pmatrix},\quad
\boldsymbol{d}^{(3)} = 
\begin{pmatrix}
2\\ 1
\end{pmatrix}
}
$$
and $n$ is an arbitrary integer; the matrix $Z$ generates the group of integral isometries.

Using this explicit description of solutions and transformation \eqref{eq:S}, one can check that for any solution of \eqref{eq:integer_eq} with $\varepsilon = -1$ at least one pair $(c_1,d_1)$ or $(c_2,d_2)$ is not coprime. At the same time, all solutions for $\varepsilon = 1$ are admissible.
\vspace{1cm}\\
\emph{3. Solutions are equivalent.}\\
Let us fix a basis for $\pi^{-1}(P)$ and move it along the path $\gamma$; this way, we get a new basis obtained by applying the monodromy matrix $M$.

We have already described all solutions of system \eqref{eq2:integer_system}. Now we need to find solutions that are conjugate by $M$.

It can be checked by direct calculation that if a vector $(d_1, d_2)$ corresponds to a pair $(M_1, M_2)$, then the vector
$$
\displaystyle{ P \begin{pmatrix} d_1 \\ d_2 \end{pmatrix},\quad P = \begin{pmatrix}
-8 & 3\\
-3 & 1
\end{pmatrix}
}
$$
corresponds to the pair $(M^{-1} M_1 M, M^{-1} M_2 M)$.

Using the following identity
$$\displaystyle{ P^{3} = -Z^{2} },$$
it is easy to show that any solution of \eqref{eq:integer_eq} can be expressed in the form 
$$
\displaystyle{
\pm P^n \begin{pmatrix}
1\\ 1
\end{pmatrix} \text{\normalfont{ or }} \pm P^n \begin{pmatrix}
1\\ 2
\end{pmatrix}\quad \text{\normalfont{ for some }} n \in \mathbb{Z}.
}
$$
Note that solutions for $\boldsymbol d$ and $-\boldsymbol d$ correspond to the same monodromy matrix.
Thus we have only two solutions up to the transformation $P$.
\end{proof}
\begin{cor}
The regular part of an almost toric fibration on $\mathbb{CP}^2$ with two focus-focus points is unique up to a fiber preserving homeomorphism.
\end{cor}
\begin{proof}
The regular part is a torus bundle over a twice punctured disk. Let $(\gamma_1,\gamma_2)$ be a pair of paths going around the punctures.
It follows from Proposition \ref{thm:prop_2pts} that there exists a basis such that the pair of the monodromy matrices corresponding to our paths has the form $(M_1^{(1)}, M_2^{(1)})$ or $(M_1^{(2)}, M_2^{(2)})$.

The pair $(M_1^{(1)}, M_2^{(1)})$ is not conjugate to $(M_1^{(2)}, M_2^{(2)})$, but the pairs $(M_1^{(1)}, M_2^{(1)})$ and $(M_2^{(2)}, M_1^{(2)})$ are conjugate. Suppose $f$ is a map such that it takes $(\gamma_1,\gamma_2)$ to $(\gamma_2,\gamma_1)$; then there exists a lift of $f$ being a fiberwise homeomorphism between the given two bundles.
\end{proof}
\begin{prop}
The regular part of an almost toric fibration on $\mathbb{CP}^2$ with three focus-focus points is unique up to a fiber preserving homeomorphism.
\end{prop}
\begin{proof}
The proof is very similar to the proof of Proposition \ref{thm:prop_2pts}.

Let $\gamma$ be the simple path going along the boundary of the base. We claim that there exists a basis such that the monodromy matrix corresponding to $\gamma$ takes the form
$$
\displaystyle{ M = \begin{pmatrix} 1 & 0\\ 9 & 1 \end{pmatrix}. }
$$

This can be proved in a similar way as we have done in Proposition \ref{thm:prop_2pts}. The preimage of the boundary of the base is a symplectic torus; the self-intersection index of this torus equals 9. It follows easily that the matrix $M$ takes the form above in a suitable basis.

Suppose $\gamma_1$, $\gamma_2$ and $\gamma_3$ is a triple of paths going around the punctures and such that $\gamma_1 \gamma_2 \gamma_3 = \gamma$.

Let us describe pairs of integer vectors $(\boldsymbol c, \boldsymbol d)$ that satisfy
$$
\displaystyle{ M_3 M_2 M_1 = M,
}
$$
$$
\displaystyle{\quad M_i = E + \varepsilon_i \begin{pmatrix} c_i d_i & d_i^2\\ -c_i^2 & -c_i d_i \end{pmatrix},\ \varepsilon_1 = \varepsilon_2 = \varepsilon_3 =\varepsilon = \pm 1.}
$$
It can be shown that this system has no solutions in the case $\varepsilon = -1$; any solution of the system for $\varepsilon = 1$ is a solution of the following one
\begin{equation}\label{eq:vector_eq}
\displaystyle{
[\boldsymbol c, \boldsymbol d] = 3 \begin{pmatrix} d_1 \\ d_2 - 3\, d_1 d_3 \\ d_3 \end{pmatrix},
}
\end{equation}
or it can be obtained from a solution of \eqref{eq:vector_eq} by the transformation $(\boldsymbol c, \boldsymbol d) \mapsto (-\boldsymbol c, -\boldsymbol d)$; here $[\boldsymbol c, \boldsymbol d]$ is the vector product of $\boldsymbol c$ and  $\boldsymbol d$. Since the latter transformation preserves the form of the monodromy matrices, it suffices to consider solutions of \eqref{eq:vector_eq}.

Equation \eqref{eq:vector_eq} may be solved for $\boldsymbol c$ if and only if 
\begin{equation}\label{eq:markov}
\displaystyle{
\langle [\boldsymbol c, \boldsymbol d], \boldsymbol d \rangle = d_1^2 +d_2^2 + d_3^2 -3\, d_1 d_2 d_3 = 0.
}
\end{equation}
This well-known equation is called \emph{Markov's equation}; let us describe its solutions. 
First note that it is sufficient to find only positive (each $d_i > 0$) solutions. Indeed, the transformation $(c_i, d_i) \mapsto (-c_i, -d_i)$ preserves matrices $M_i$. Further, one can change the sign of two components of $\boldsymbol d$ using the symmetry of Markov's equation.

Further, Markov's equation has the following \emph{elementary} symmetries:
\begin{itemize}
\item assume that $(d_1, d_2, d_3)$ is a solution of \eqref{eq:markov}; then one can obtain new solutions by permutation $(d_1, d_2, d_3)$;

\item if $(d_1, d_2, d_3)$ is a solution of \eqref{eq:markov}; then we can get another one by using the following transformation
$$
\displaystyle{ v(\boldsymbol d) = (d_1, 3\, d_1 d_3 - d_2, d_3). }
$$
Note that $v^2 = \text{id}$.

\end{itemize}
It is well known that any solution of \eqref{eq:markov} can be obtained from $(1,1,1)$ by a sequence of the elementary symmetries.
\begin{center}
\begin{forest}
[(1 1 1) [(1 1 2)[(1 2 5)[(1 5 13) [(1 13 34)][(5 13 194)]][(2 5 29) [(5 29 433)][(2 29 169)]]]]]
\node at (current bounding box.south)[below=3ex]
{The first levels of the Markov number tree.};
\end{forest}
\end{center}
For any solution $\boldsymbol d$ its components are coprime; it follows that any solution of system \eqref{eq:vector_eq} with respect to $\boldsymbol c$ has the form
\begin{equation}\label{eq:vector_sol}
\displaystyle{ \boldsymbol c + k\, \boldsymbol d,\ k \in \mathbb{Z},\ \text{here  $\boldsymbol c$ is a solution of \eqref{eq:vector_eq}}. }
\end{equation}
Let us consider the matrix
$$
\displaystyle{ A = \begin{pmatrix}
1 & 0\\
1 & 1
\end{pmatrix}. }
$$
Note that $M$ and $A$ commute.
Therefore, if a triple of matrices $M_i$ is a solution, then the triple $A^{-1} M_i A$ is also a solution. Assume that a pair $(\boldsymbol c, \boldsymbol d)$ corresponds to a triple $M_i$, then $(\boldsymbol c + \boldsymbol d, \boldsymbol d)$ corresponds to $A^{-1} M_i A$. Therefore, the conjugacy class of a triple $M_i$ is uniquely determined by the corresponding solution of Markov's equation.

We see that for any triple $(d_1, d_2, d_3)$ there exists a unique torus bundle. Now, let us show that these bundles are fiberwise homeomorphic.

Let us consider a symmetry of Markov's equation and construct an automorphism $\gamma_i \mapsto \gamma_{i}^{\prime}$ of the fundamental group of our base that induces that symmetry.

\begin{enumerate}
\item \emph{cyclic permutation} $(d_1,d_2,d_3) \mapsto (d_3,d_1,d_2)$\\
It can be induced by the following automorphism
$$
\displaystyle{ \gamma_{1}^{\prime} = \gamma \gamma_{3} \gamma^{-1},\ \gamma_2^{\prime} = \gamma_{1},\ \gamma_3^{\prime} = \gamma_2.}
$$
\item \emph{reflection} $(d_1,d_2,d_3) \mapsto (d_3,d_2,d_1)$\\
It can be obtained as follows.
First let us consider the automorphism
$$
\displaystyle{ \gamma_{1}^{\prime} = \gamma_3^{-1},\ \gamma_{2}^{\prime} = \gamma_2^{-1},\ \gamma_{3}^{\prime} = \gamma_1^{-1} }
$$
that takes a triple $(M_1, M_2, M_3)$ to $(M_3^{-1}\!, M_2^{-1},\! M_1^{-1})$. The product of matrices of this triple is not equal to $M$, but equals $M^{-1}$.
Now note that the matrices $M$ and $M^{-1}$ are conjugate
$$
\displaystyle{ C^{-1} M C = M^{-1},\quad C = \begin{pmatrix}
-1 & 0\\
0 & 1
\end{pmatrix} }
$$
and any matrix $M_i$ can be obtained from the matrix $C^{-1} M_i C$ by the following transformation 
$$(c_i, d_i) \mapsto (-c_i, d_i).$$
\item \emph{Vieta jumping} $(d_1,d_2,d_3) \mapsto (d_1,3\,d_1 d_3 - d_2, d_3)$\\
The automorphism
$$
\displaystyle{ \gamma_{1}^{\prime} = \gamma_1 \gamma_2 \gamma_1^{-1},\ \gamma_{2}^{\prime} = \gamma_1,\ \gamma_{3}^{\prime} = \gamma_3}
$$
induces the transformation
$(d_1, d_2, d_3) \mapsto (3\,d_1 d_3 - d_2, d_1, d_3)$. 
Now one can use the permutation of the first two components.
\end{enumerate}

For any transformation described above, one can easily check that the corresponding transformation of $\boldsymbol c$ is integral. We do not need the explicit formulas, so they are omitted. Since equation \eqref{eq:vector_eq} has a solution for $\boldsymbol d = (1,1,1)$, it follows that it has a solution for arbitrary $\boldsymbol d$ satisfying \eqref{eq:markov}.

Finally, it is not hard to check using Dehn twist (see \cite{I}) that any automorphism described above is induced by some homeomorphism of the base space. This implies that bundles corresponding to distinct solutions of \eqref{eq:markov} are fiberwise homeomorphic.
\end{proof}

The author is grateful to Andrey Oshemkov and Anton Izosimov for useful suggestions.

This research was partially supported by the RFBR grant N.\,14-01-00119 and by the PRIN 2010-11 grant \enquote{Geometric and analytic theory of Hamiltonian
systems in finite and infinite dimensions}, of Italian Ministry of Universities and Researches.

\bibliographystyle{elsarticle-num}

\bibliography{research}

\begin{comment}

\end{comment}

\end{document}